\documentclass{amsart}

\usepackage{amsthm,amsmath,amssymb}
\usepackage{graphicx,float}
\usepackage[colorlinks,
            linkcolor = blue,
            urlcolor  = blue,
            citecolor = blue,
            anchorcolor = blue]{hyperref}

\newcommand{\stat}{\operatorname{stat}}
\newcommand{\area}{\operatorname{area}}
\newcommand{\dinv}{\operatorname{dinv}}
\newcommand{\ides}{\operatorname{ides}}
\newcommand{\comp}{\operatorname{comp}}

\DeclareMathOperator{\Cop}{C}
\DeclareMathOperator{\Sop}{S}
\DeclareMathOperator{\qpoly}{qpoly}

\newtheorem{theorem}{Theorem}[section]
\newtheorem{lemma}[theorem]{Lemma}
\newtheorem{conjecture}[theorem]{Conjecture}
\newtheorem{corollary}[theorem]{Corollary}

\theoremstyle{remark}
\newtheorem{remark}[theorem]{Remark}

\numberwithin{equation}{section}

\title{ A combinatorial model for \( \nabla m_\mu \)}\protect\thanks{An extended abstract will appear in the DMTCS Proceedings of FPSAC 2018.} 
\author{Emily Sergel} \thanks{Partially supported by NSF grant DMS-1603681.}
\address{Department of Mathematics\\ University of Pennsylvania \\ Philadelphia, PA, USA}
\email{esergel@math.upenn.edu}

\begin{document}

\maketitle

\begin{abstract}

The modified Macdonald polynomials introduced by Garsia and Haiman (1996) have many remarkable combinatorial properties. One such class of properties involves applying the \(\nabla\) operator of Bergeron and Garsia (1999) to basic symmetric functions. The first discovery of this type was the Shuffle Conjecture of Haglund, Haiman, Loehr, Remmel, and Ulyanov (2005), which relates the expression \(\nabla e_n\) to parking functions. A refinement of this conjecture, called the Compositional Shuffle Conjecture, was introduced by Haglund, Morse, and Zabrocki (2012) and proved by Carlsson and Mellit (2015).

We give a symmetric function identity relating hook monomial symmetric functions to the operators used in the Compositional Shuffle Conjecture. This implies a parking function interpretation for nabla of a hook monomial symmetric function, as well as LLT positivity. We show that our identity is a \(q\)-analog of the expansion of a hook monomial into complete homogeneous symmetric functions given by Kulikauskas and Remmel (2006). We use this connection to conjecture a model for expanding \(\nabla m_\mu\) in this way when \(\mu\) is not a hook.

\end{abstract}

\section{Introduction} \label{sec:newintro}

Jeff Remmel was an important figure in the world of algebraic combinatorics. He was well-loved in this community, especially by his many students, including those he did not formally advise. The author was one such student, substantially shaped as a mathematician by all that she learned from him during her graduate studies. He will be dearly missed.

Jeff's research included key classical results in symmetric function theory, such as the combinatorial descriptions for the transition matrices between the standard bases. Many of these clever and subtle constructions are due to E\(\mathrm{\breve{g}}\)ecio\(\mathrm{\breve{g}}\)lu and Remmel \cite{rimhook,bricktab}. Beck, Remmel and Whitehead \cite{Bntrans} filled in the remaining transition matrices without proof while extending these rules to the \(B_n\)-case. Kulikauskas and Remmel \cite{bibrick} provided the missing details.

A complete table of transition matrices is an invaluable tool in symmetric function theory. These basis changes often pop up in other problems requiring quasi-symmetric or \(q\)-analogs. The problem discussed here is one such example. Observing experimentally that \(\nabla m_\mu\) is always Schur positive or Schur negative has lead the author to study a \(q\)-enumeration of Beck, Remmel and Whitehead's objects, called bi-brick permutations. Throughout we follow the conventions of Kulikauskas and Remmel. 


\section{Haglund, Morse, and Zabrocki's \( \Cop \) operators} \label{Intro}

In 1988, Macdonald \cite{macoriginal} introduced a new basis for the ring of symmetric functions. (See Macdonald \cite{macbook} for an introduction to symmetric function theory.) Later Garsia and Haiman \cite{modmac} modified this basis to form the modified Macdonald polynomial basis \( \{ \widetilde{H}_\mu[X;q,t] \} \). They sought a representation-theoretic interpretation for this basis, which led them to study a number of remarkable \(S_n\) bi-modules. Among these was the module of Diagonal Harmonics. They conjectured a formula for its Frobenius characteristic \(DH_n[X;q,t]\) and Haiman \cite{DHdim} later proved their conjecture using algebraic geometry. However, this formula is not obviously Schur positive or even polynomial.

Bergeron and Garsia \cite{SciFi} noted that the formula of Garsia and Haiman was very close to the modified Macdonald expansion of \(e_n\). Inspired by this similarity, they defined the linear symmetric function operator \( \nabla \), which acts by \( \nabla \widetilde{H}_\mu = t^{n(\mu)} q^{n(\mu')} \widetilde{H}_\mu \). In this language, \( DH_n[X;q,t]=\nabla e_n \). In \cite{shuffleconj}, Haglund, Haiman, Loehr, Remmel and Ulyanov discovered a combinatorial interpretation for \( \nabla e_n \) in terms of parking functions (defined below). Their conjecture is known as the Shuffle Conjecture, and was only recently proved.

In 1966, Konheim and Weiss \cite{KW} introduced parking functions to study a combinatorial problem involving cars parking on a one-way street. While they thought of parking functions as functions, for our purposes it is more helpful to follow the interpretation introduced by Garsia and Haiman \cite{PFpath}. A Dyck path in the \( n \times n \) lattice is a path \( (0,0) \) to \( (n,n) \) of North and East steps which stays weakly above the line \(y=x\). A parking function is a Dyck path with labels \( \{1,2,\dots,n\} \) on North steps which are column-increasing. We write the labels of a parking function in the cell just East of each North step. The labels of a parking function are known as cars. For example, see Figure~\ref{fig:PFex}. 

\begin{figure}[H]
\begin{center}
\includegraphics[width=1.2in]{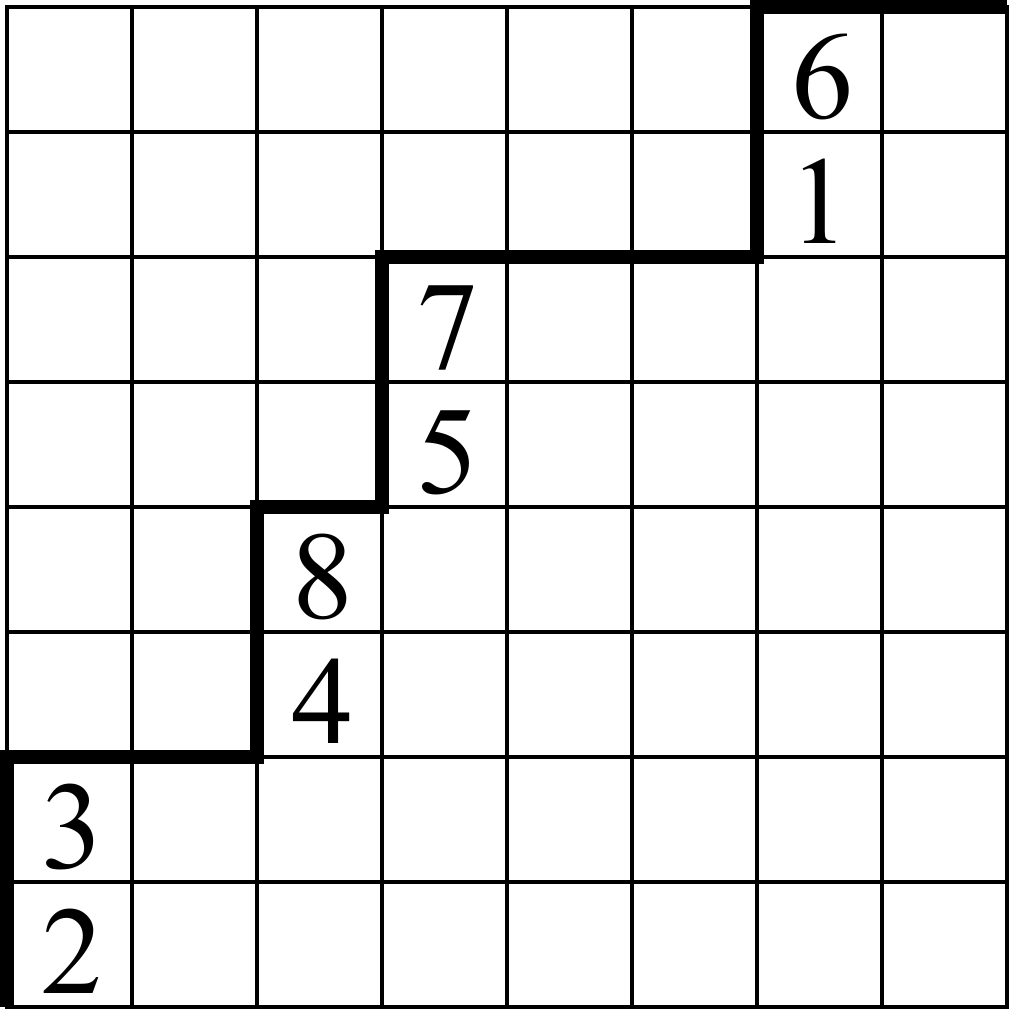}
\caption{A parking function with 8 cars.}
\label{fig:PFex}
\end{center}
\end{figure}

The symmetric function \( \nabla e_n \) is a weighted sum of parking functions involving three statistics. The most natural of these statistics is the \( \area \) - the number of full cells between the main diagonal \( y=x \) and the underlying Dyck path. In Figure~\ref{fig:PFex}, the \( \area \) is 6. The other two statistics use the notion of diagonals. Let the \( k \)-diagonal be the set of cells cut by the line \( y=x+k \). In particular, the main diagonal or \(0\)-diagonal consists of the cells cut by \( y=x \). In Figure~\ref{fig:PFex}, there are 3 cars in the \( 0 \)-diagonal, 4 cars in the \( 1 \)-diagonal, and 1 car in the \( 2 \)-diagonal.

The word \( \sigma \) of a parking function is the permutation obtained by reading cars from highest to lowest diagonal and right to left within each diagonal. In Figure~\ref{fig:PFex}, \( \sigma = 7 \, 6 \, 5 \, 8 \, 3 \, 1 \, 4 \, 2 \). Recall that the \( \ides \) of a permutation \( \sigma \) is the descent set of \( \sigma^{-1} \). Alternatively, it is the set of \( i \) so that \( i+1 \) occurs left of \( i \) in \( \sigma \). In the example, \( \ides(\sigma) = \{2,4,5,6\} \).  Then each parking function \( PF \) is weighted by the quasi-symmetric function \( F_{\ides(PF)} \). Here if \( S \subset \{1,2,\dots,n-1\} \), \( F_S \) is the following degree-\( n \) fundamental quasi-symmetric function defined by  Gessel \cite{Gessel}.
\[
F_S = \sum_{\substack{ 0 \leq a_1 \leq a_2 \leq \dots \leq a_n \\ i \in S \Rightarrow a_i < a_{i+1}}} x_{a_1} x_{a_2} \dots x_{a_n}
\]

Finally,  the \( \dinv \) of a parking function counts certain inversions in \( \sigma \). If two cars are in the same diagonal and the larger occurs further right, we say they create a primary diagonal inversion. If two cars are in adjacent diagonals so that the larger car is higher and further left, they create a secondary diagonal inversion. The \( \dinv \) of a parking function is the total number of primary and secondary diagonal inversions. In Figure~\ref{fig:PFex}, for example, cars 3 and 5 make a primary diagonal inversion, while cars 1 and 3 make a secondary diagonal inversion. In total, there are five primary diagonal inversions and four secondary diagonal inversions in our example. Hence \( \dinv=9 \).

Let \( {\mathcal PF}_n \) be the set of all parking functions on \( n \) cars. Then the classical Shuffle Conjecture of Haglund, Haiman, Loehr, Remmel and Ulyanov \cite{shuffleconj} states
\[
\nabla e_n = \sum_{PF \in {\mathcal PF}_n} t^{\area(PF)} q^{\dinv(PF)} F_{\ides(PF)}.
\]
In \cite{compconj}, Haglund, Morse and Zabrocki refined the Shuffle Conjecture using the following plethystic symmetric function operators, \( \Cop_a \) for non-negative integers \(a\). For any symmetric function \( P\), 
\[
\Cop_a \, P[X] \, = \, \left(- \frac{1}{q} \right)^{a-1} P \hspace{-.1cm} \left[ X - \frac{1-1/q}{z} \right] \, \sum_{m \geq 0} z^m h_m[X] \, \Big|_{z^a}
\]
where \(f|_{z^a}\) is the coefficient of \(z^a\) when \(f\) is expanded as a formal power series in \(z\). (For an introduction to plethystic notation see Loehr and Remmel \cite{pleth}.) Their refinement of the Shuffle Conjecture, which is stated below, was recently proved by Carlsson and Mellit \cite{compproof}. Here \( \comp(PF) \) is the composition of \( n \) giving the distances between points \( (i,i) \) on \( PF \)'s underlying path. For example, the parking function in Figure~\ref{fig:PFex} has \( \comp = (2,4,2) \). 
\begin{theorem}[Carlsson-Mellit]  \label{thm:comp}
For all compositions \( \alpha \models n \),
\[
\nabla \Cop_{\alpha_1} \cdots \Cop_{\alpha_{\ell(\alpha)}} 1 = \sum_{\substack{ PF \in {\mathcal PF}_n \\ \comp(PF)=\alpha}} t^{\area(PF)} q^{\dinv(PF)} F_{\ides(PF)}.
\]
\end{theorem}

We will use the shorthand \( \Cop_\alpha = \Cop_{\alpha_1} \circ \cdots \circ \Cop_{\alpha_k} \) for a composition \( \alpha=(\alpha_1,\dots,\alpha_{\ell(\alpha)}) \). Haglund, Morse and Zabrocki \cite{compconj} showed that
\[
e_n = \sum_{\alpha \models n} \Cop_{\alpha} 1.
\]
This, together with the Compositional Shuffle Conjecture implies the classical Shuffle Conjecture.

Tthe Compositional Shuffle Conjecture can be used as a tool for finding and proving combinatorial interpretations of images under the \( \nabla \) operator: If some symmetric function \( f \) can be expanded positively using the \( \Cop \) operators applied to 1, then \( \nabla f \) can be interpreted as a weighted sum of parking functions. Additionally, Haglund, Haiman, Loehr, Remmel and Ulyanov showed that the weighted sum of all parking functions with the same supporting Dyck path can be interpreted as an LLT polynomial. These polynomials, introduced by Lascoux, Leclerc and Thibon \cite{LLT} are well-studied symmetric functions that are believed to be Schur-positive. Indeed, Grojnowski and Haiman have proven the positivity conjecture in an unpublished manuscript \cite{LLTpos}. (It is an open problem to explicitly give the Schur expansion of an LLT polynomial.) Hence a positive ``\( \Cop \) expansion" of \( f \) implies the Schur positivity of \( \nabla f \). 

However, the family \( \{ \Cop_\alpha 1 \}_{\alpha \models n} \) does not form a basis for the ring of symmetric functions - the subcollection \( \{ \Cop_\lambda 1 \}_{\lambda \vdash n} \) does. But simply expanding \(f\) in terms of this basis may not yield a positive or even polynomial expansion, despite the existence of a nice expansion into the full collection. We give a simple criteria for finding \( \Cop \) expansions which was pointed out to the author by Adriano Garsia.

\begin{lemma}
For any composition \( \alpha \),
\[ \Cop_\alpha 1 \Big|_{q=1} = (-1)^{|\alpha|-l(\alpha)} \, h_{\alpha} \]
\end{lemma}

\begin{proof}
Note that when \(q=1\), the plethystic shift in the definition of the \( \Cop \) operators disappears. Hence we are just left with
\[
\Cop_k f \Big|_{q=1} = (-1)^{k-1} \, h_k \, f.
\vspace{-18pt}
\]
\end{proof}
\noindent By the lemma, a symmetric function \(f\) can only have a positive \( \Cop \)-expansion if its expansion  \( f = \sum_{\lambda} c_{\lambda} h_{\lambda} \) into the \(h\)-basis has integral coefficients \( c_\lambda \) with sign \( (-1)^{|\lambda|-\ell(\lambda)} \). If \( f \) does have such a \( \Cop \)-expansion, then the coefficients of all \( \alpha \) rearranging to \( \lambda \) sum to a \( q \)-analog of \( |c_\lambda| \).

Experimentally we can see that \( \nabla m_\mu \) is always Schur positive or Schur negative. As part of her thesis work, the author and Garsia \cite{mythesis} showed that
\[
(-1)^{n-1} p_n = \sum_{\alpha \models n} [\alpha_n]_q \Cop_\alpha 1.
\]
Inspired by the fact that \( m_{1^n}=e_n \) and \( m_n=p_n \), we seek a positive polynomial \( \Cop \) expansion for the monomial symmetric functions. We succeed in doing this for \( m_{\mu} \) when \( \mu \) is a hook shape. Based on the criteria above, we see that any such expansion must be a \(q\)-analog of the combinatorial expansion for monomial symmetric functions into the complete homogeneous symmetric functions given by Kulikauskas and Remmel \cite{bibrick}. Such a \(q\)-analog has two important components: a \(q\)-statistic and a refinement of the associated partition (previously recorded by \(h_\lambda\)) to a composition (now recorded by \( \Cop_\alpha \)). We find a specific candidate for the later and make some conjectural observations about the former.


\section{Bi-brick permutations and a model for \(m_\mu\)}

We begin with a brief overview of the combinatorial expansion from the monomial symmetric functions to the complete homogeneous symmetric functions as described in \cite{bibrick}. Essentially, the coefficient of \( h_\lambda \) in \( m_\mu \) counts (with the appropriate sign) what Beck, Remmel, and Whitehead \cite{Bntrans} call ``bi-brick permutations." These objects are analogous to permutations written in cycle notation. They consists of products of cycles decorated inside and outside with ``bricks". The lengths of the inner bricks form the parts of \( \mu \), while the outer bricks form \( \lambda \). Rotational symmetry is forbidden. For example, the object in Figure~\ref{fig:goodex} is a bi-brick permutation. The cycle shown in Figure~\ref{fig:badex}, on the other hand, is not allowed in any bi-brick permutation. This is because it is unchanged when rotated 180 degrees.

\begin{figure}
\begin{center}
\includegraphics[width=3in]{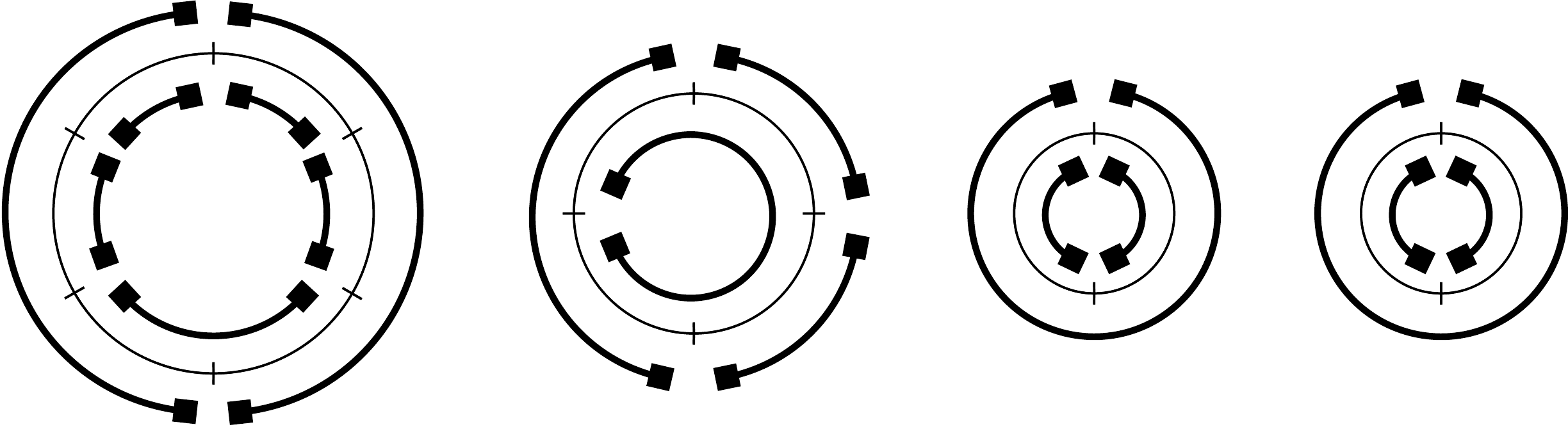}
\caption{A bi-brick permutation of size 14.}
\label{fig:goodex}
\end{center}
\end{figure}

\begin{figure}
\begin{center}
\includegraphics[width=.8in]{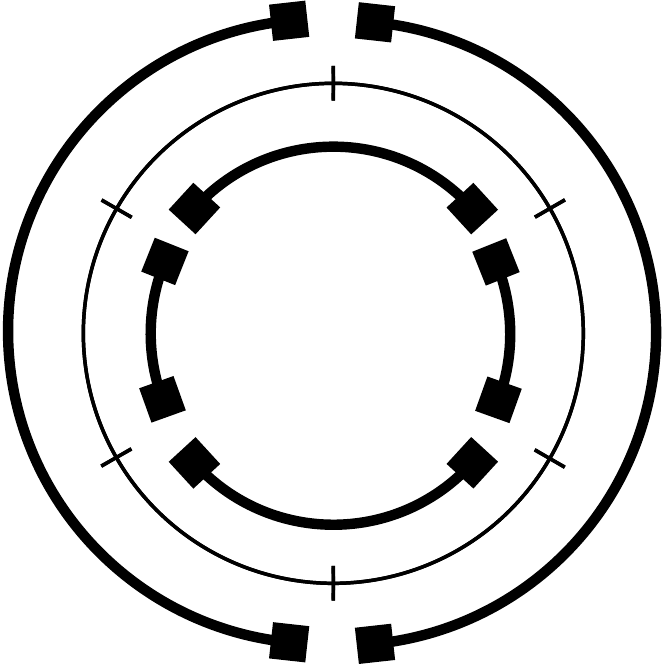}
\caption{A cycle with rotational symmetry.}
\label{fig:badex}
\end{center}
\end{figure}

Bi-brick permutations are in bijection with multisets of Lyndon words in the alphabet $\{ B < L < N < U\}$. (Here we deviate slightly from \cite{bibrick} by using the letter ``U'' in place of ``M''. This is to avoid confusion about the alphabetical order.) The bijection is as follows: Consider each cycle individually. Suppose the total length of the cycle is $n$. Start at any of the $n$ segments with $w$ initialized to the empty word and work clockwise. At each segment, add a letter to the end of $w$ according to which kind of brick(s) start at that segment: $B$ if both kinds start, $L$ if only an outer brick (contributing to $\lambda$) starts, $U$ if only an inner brick (contributing to $\mu$) starts, and $N$ if neither kind of brick starts. When the cycle is complete, $w$ will be a word of length $n$. Since the initial cycle has no rotational symmetry, $w$ is not a power of a shorter word. Hence there is exactly one Lyndon word in its rotational orbit. Take this Lyndon word to the be image of the given cycle. This process is clearly invertible since the letters of the word describe the starting cycle. Applying this process to each cycle of the bi-brick permutation in Figure~\ref{fig:goodex} gives $\{BUULUU,LLLU,BU,BU\}$. However, for the cycle in Figure~\ref{fig:badex} we obtain $LUULUU$ which has no Lyndon word in its rotational orbit. This is due to the rotational symmetry of the starting cycle.

For any bi-brick permutation $\Pi$, let $\mu(\Pi)$ be the partition whose parts are the lengths of $\Pi$'s inner bricks. We also associate a composition to $\Pi$ whose parts are the lengths of $\Pi$'s outer bricks read in a specific order. That is, for a cycle $C$ of length $n$ whose Lyndon word contains a $B$, let $\alpha(C) = (i_2-i_1,i_3-i_2, \dots, n+i_1-i_k)$ where $i_1,i_2,\dots,i_k$ are the locations of all $B$'s and $L$'s in $C$'s Lyndon word. (Note: $i_1$ is always 1.) If $C$ is a cycle whose Lyndon word does not contain a $B$, rotate the inner blocks of $C$ clockwise as many times as necessary until the resulting cycle $\widehat{C}$ has a $B$ in its Lyndon word (i.e. until the ends of some inner and outer bricks line up.) Let $\alpha(C)=\alpha(\widehat{C})$. Finally for a bi-brick permutation $\Pi$, let $\alpha(\Pi)$ be the composition obtained by concatenating $\alpha(C)$ for all $C \in \Pi$ in reverse lexicographic order of Lyndon words. For example, the Lyndon word of the second cycle of Figure~\ref{fig:goodex} does not contain a $B$. Rotating gives a new cycle with Lyndon word $BLLN$. Therefore the cycles of Figure~\ref{fig:goodex} have $\alpha$ equal to $(3,3)$, $(1,1,2)$, $(2)$, and $(2)$, respectively, from left to right. Sorting them according to the Lyndon words of the original cycles gives $\alpha(\Pi)=(1,1,2,3,3,2,2)$.

\begin{figure}
\begin{center}
\includegraphics[width=2in]{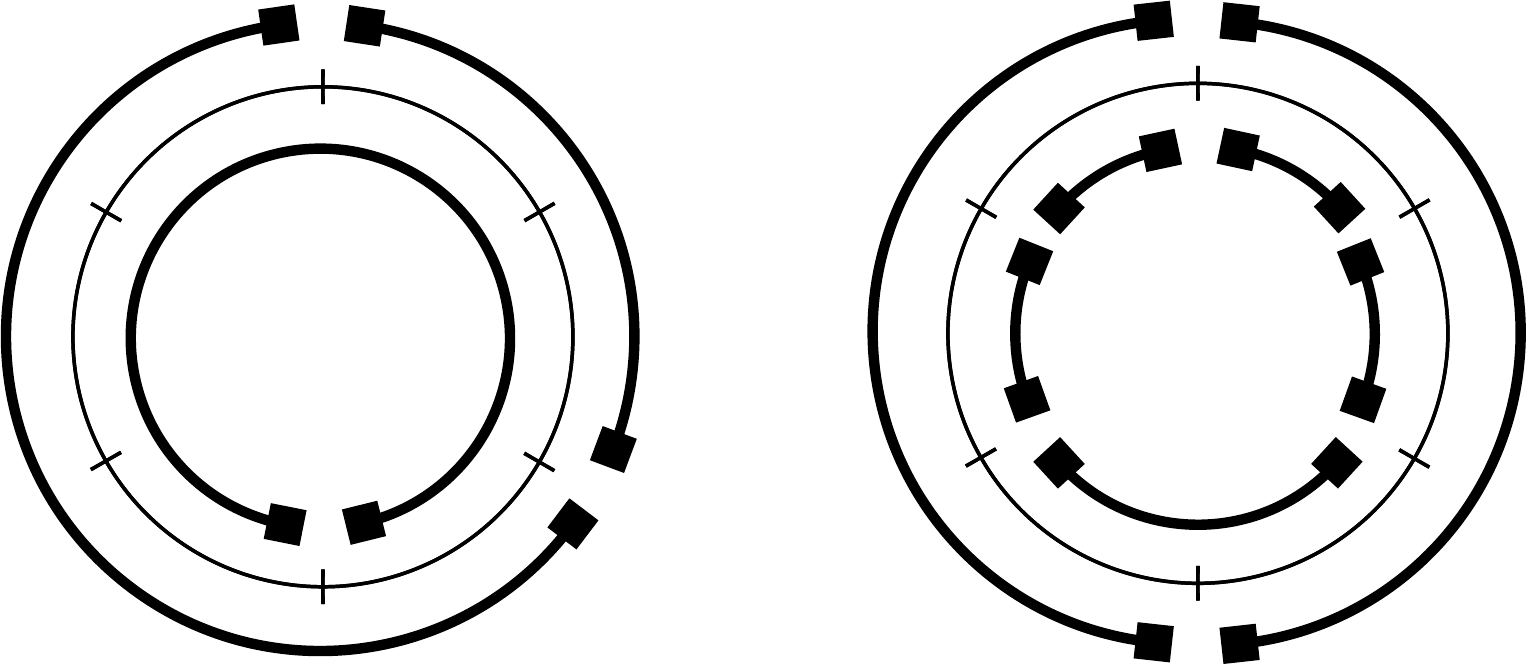}
\caption{A bi-brick permutation with size 12.}
\label{fig:goodex2}
\end{center}
\end{figure}

It is important that when a cycle $C$'s Lyndon word does not contain $B$, the order of $\alpha$ is read according to the Lyndon word of the rotated cycle $\widehat{C}$, yet $C$'s Lyndon word is used for sorting. For example, in Figure~\ref{fig:goodex2}, the first cycle $C$ has Lyndon word $LNLUNN$ giving $\alpha=(4,2)$ while $\widehat{C}$'s Lyndon word $BNLNNN$ gives $\alpha=(2,4)$. The Lyndon word of the second cycle is $BUULUU$ which lies lexicographically between the Lyndon words of $C$ and $\widehat{C}$. So for this bi-brick permutation $\Pi$, we have $\alpha(\Pi)=(2,4,3,3)$.

Based on experimental data, we make the following conjecture.
\begin{conjecture}
Let $\mu$ be any partition. For each bi-brick permutation $\Pi$ with $\mu(\Pi)=\mu$, there is a non-negative integer $\stat(\Pi)$ so that
\[
(-1)^{|\mu|-\ell(\mu)} m_{\mu} = \sum_{\mu(\Pi)=\mu} q^{\stat(\Pi)} \Cop_{\alpha(\Pi)} 1.
\]
\end{conjecture}
\noindent Furthermore, it seems from the data that we can always do this in a way such that
\begin{itemize}
\item $\stat(\Pi)$ is the sum of $\stat(C)$ for the individual cycles $C$ of $\Pi$ (with multiplicity),
\item $\stat(C)=0$ for a single cycle $C$ if and only if the Lyndon word corresponding to $C$ contains the letter $B$, and
\item $\stat(C)<n$ for every cycle $C$ of length $n$.
\end{itemize}

\begin{remark}
Note that if our conjecture is true, then every ``monomial alternating'' symmetric function has a positive $\Cop$ expansion. This condition is sufficient but not necessary. For example,
\[
q^2 \Cop_3 1 = h_3 = m_3 + m_{2,1} + m_{1,1,1}.
\]
Furthermore, not every $h_{\lambda}$ has the form $\pm q^m \Cop_{\alpha} 1$ for some power $m$ and rearrangement $\alpha$ of $\lambda$. I.e.,
\[
-q \Cop_{(2,1)} 1 = -q^2 \Cop_{(1,2)} 1 = h_{(2,1)} - \frac{q-1}{q} h_3
\]
So it is still an open problem to find (even conjecturally) a simple algorithm to check for $\Cop$ positivity. 
\end{remark}


\section{A \(C\) expansion for hook monomials} \label{sec:hook}

In this section, we present a positive $\Cop$ expansion for $(-1)^{|\mu|-\ell(\mu)} m_{\mu}$ when $\mu$ is a hook shape. First, the coefficient of each $\Cop$ operator is a polynomial in $q$ computed algorithmically. We prove this formula inductively. Then we will see that this polynomial enumerates bi-brick permutations according to our conjecture.

\begin{theorem} \label{thm:mexp}
Let \( n \geq 2 \) and \( k \geq 1 \). Then
\[
(-1)^{n-1} m_{n,1^k} = \sum_{a=1}^n \left(k+1 + \sum_{i=1}^{a-1} q^{n-i}\right) \sum_{\tau \models n-a} \, \sum_{b=0}^k \, \sum_{\rho \models k-b} \Cop_\tau \Cop_{a+b} \Cop_{\rho} 1
\]
\end{theorem}

Applying \(\nabla\) to this identity, together with the Compositional Shuffle Conjecture, gives a parking function interpretation for \( (-1)^{n-1} \nabla m_{n,1^k} \).

\begin{corollary} \label{cor:PFexp}
Let \( n \geq 2 \) and \( k \geq 1 \). Then
\[
(-1)^{n-1} \nabla m_{n,1^k} = \sum_{PF \in {\mathcal PF}_n} \qpoly_{n,k}(PF) \, t^{\area(PF)} q^{\dinv(PF)} F_{ides(\sigma(PF))}
\]
where \( \qpoly_{n,k}(PF) \) is computed as follows. The constant term is \( k+1 \). If \(PF\) is not touching the main diagonal at \( (n-1,n-1) \), add \( q^{n-1} \). Otherwise stop. If \( PF \) is not touching the main diagonal at \( (n-2,n-2) \), add \( q^{n-2} \). Otherwise stop. Continue in this way until you stop or run out of points on the diagonal (this does not include the starting point).
\end{corollary}

For example, if \( n=5 \) and \( k=3 \), then the parking function \( PF \) in Figure~\ref{fig:PFex} has \( \qpoly_{5,3}(PF) = 4 + q^3 +q^4 \). This is because \( PF \) does not touch the main diagonal at \( (4,4) \) or \( (3,3) \) but it does touch at \( (2,2) \). Note that \(\qpoly_{n,k}(PF) \) does not depend on the placement of the cars in \( PF \), only on the underlying Dyck path. Hence Corollary~\ref{cor:PFexp} expresses \( \nabla m_{n,1^k} \) as a positive sum of LLT polynomials.

Before we prove this theorem, we need a lemma. In \cite{compconj}, the authors observe that the \( \Cop \) operators are closely related to the Bernstein operators: Let \( \Sop_a \) be the operator that sends \(s_{\lambda_1,\dots,\lambda_k}\) to \( s_{a,\lambda_1,\dots,\lambda_k} \) for any partition \( \lambda = (\lambda_1 \geq \dots \geq \lambda_k) \). Then for any symmetric function \( P \),
\[
\Sop_a P[X] = P \left[ X - \frac{1}{z} \right] \sum_{m \geq 0} z^m h_m[X] \Big|_{z^a}.
\]
While they use this relationship to express the \( \Sop \) operators in terms of the \( \Cop \) operators, the reverse can be accomplished in a similar way. 

\begin{lemma}
For all \( m \geq 0 \),
\[
\Cop_m = \left( \frac{-1}{q} \right)^{m-1} \sum_{k \geq 0} \frac{\Sop_{m+k} h_k^{\perp}}{q^k}.
\]
Hence for all partitions \( \lambda \),
\[
(-q)^{m-1} \Cop_m s_\lambda = \sum_{\mu} \frac{s_{m+|\lambda / \mu|, \mu}}{q^{|\lambda/\mu|}}
\]
where the sum is over all partitions \( \mu \subseteq \lambda \) for which the diagram \( \lambda / \mu \) is a horizontal strip (i.e., contains no two boxes in the same column).
\end{lemma}

\begin{proof}
In Proposition 3.6 of \cite{compconj}, Haglund, Morse and Zabrocki give the expansion
\[
\Sop_{m} = (-q)^{m-1} \sum_{i \geq 0} \Cop_{m+i} e_i^\perp.
\]
Plugging this into the right hand side of our desired equality gives
\begin{align*}
 \left( \frac{-1}{q} \right)^{m-1} \sum_{k \geq 0} \frac{\Sop_{m+k} h_k^{\perp}}{q^k} &= \sum_{i,k \geq 0} (-1)^k \Cop_{m+k+i} e_i^\perp h_k^\perp \\
 &= \sum_{d \geq 0} \Cop_{m+d} \sum_{k=0}^d (-1)^k e_{d-k}^\perp h_k^\perp.
 \end{align*}
 Since \(\sum_{k=0}^d (-1)^k e_{d-k} h_k = 0\) whenever $d>0$, the only nonzero term is $\Cop_m$. This gives the first equality. The second follows from the Pieri rule.
 
\end{proof}

\begin{proof}[Proof of Theorem~\ref{thm:mexp}]

Note that our identity is equivalent to
\begin{align*}
(-&1)^{n-1} m_{n,1^k} - (k+1) e_{n+k} \\
&= \sum_{i=1}^{n-2} q^i \Cop_i \left( ({-}1)^{n-i+1} m_{n-i,1^k} {-} (k{+}1) e_{n-i+k} \right)
+ q \, [n-1]_q \sum_{i=n}^{n+k} \Cop_i e_{n-i+k}\\
&= \sum_{i=1}^{n-2} q^i \Cop_i \left( ({-}1)^{n-i+1} m_{n-i,1^k} {-} (k{+}1) e_{n-i+k} \right)
+ (-1)^{n-1} \frac{[n-1]_q}{q^{n-2}} s_{n,1^k}
\end{align*}
Induct on \( n \) with \( k \) fixed. The above clearly holds when \( n=2 \). Now suppose \( N>2 \) and that the identity holds for all \( n < N \). A combinatorial expansion for any monomial symmetric function into the Schur basis is given by E\(\mathrm{\breve{g}}\)ecio\(\mathrm{\breve{g}}\)lu and Remmel \cite{rimhook} in terms of rim hook tabloids. However, we only use the special case of a hook, which appears in \S I.6 Exercise 4(e)(iii) of  \cite{macbook}:
\[
\langle m_{a,1^k}, s_{\mu} \rangle = \begin{cases} (-1)^{a-1} (k+1) & \hbox{ if } \mu=1^{a+k} \\
(-1)^{a-\ell} & \hbox{ if } \mu=(\ell,2^j,1^i) \hbox{ for } j+\ell \leq a \\
0 & \hbox{ otherwise.} \end{cases}
\]
This, together with the above lemma can be used to compute the Schur expansion of the right hand side. It is then routine (but technical) to show that this matches the Schur expansion of the left hand side.

\end{proof}

\begin{figure}
\begin{center}
\includegraphics[width=4in]{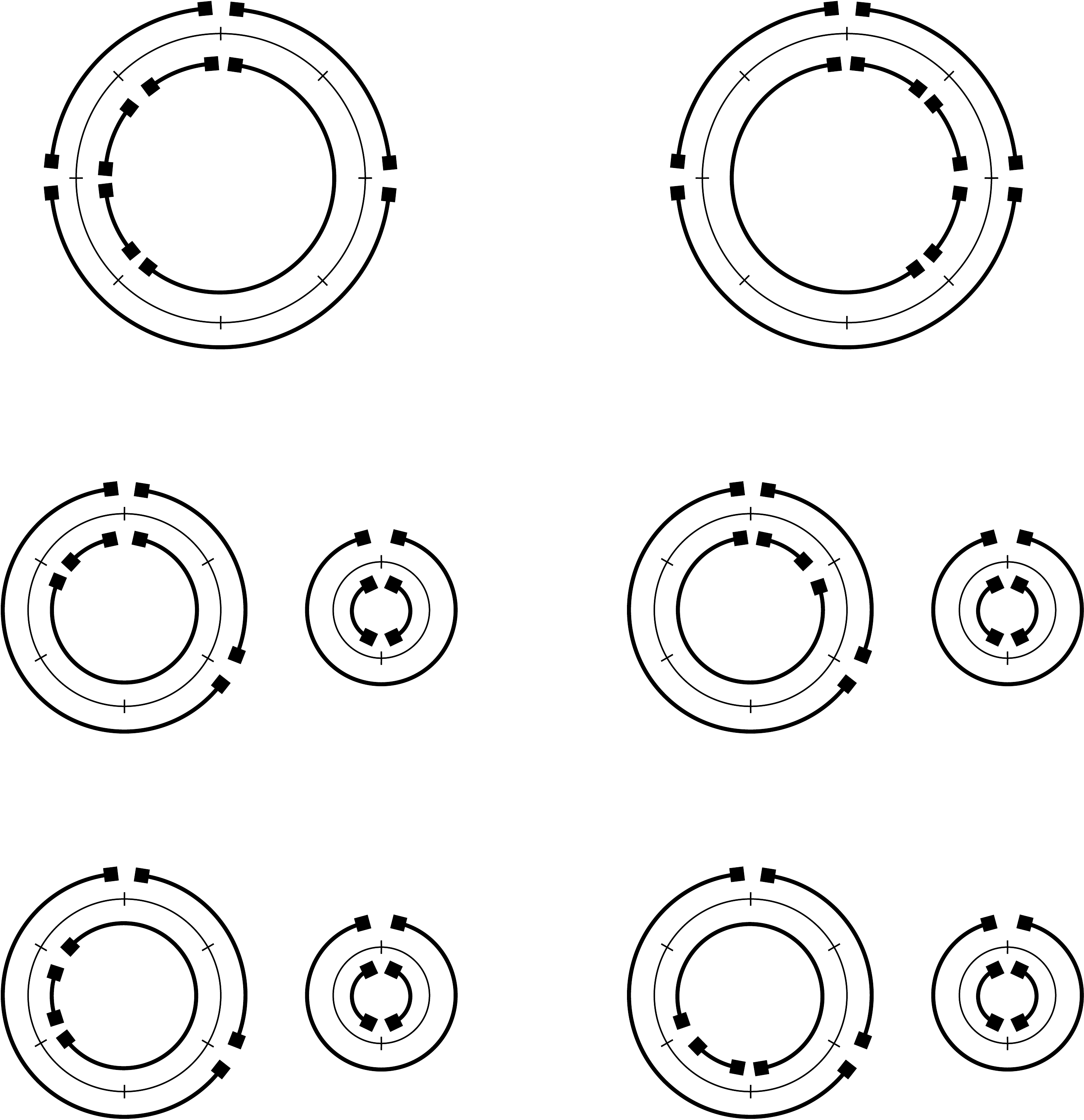}
\caption{The bi-brick permutations corresponding to Figure~\ref{fig:PFex} when $n=5$ and $k=3$.}
\label{fig:bigex}
\end{center}
\end{figure}

Now we show the correspondence between this formula and the hook case of our conjecture. Suppose that $\mu=(n,1^k)$ and let $\alpha$ be any composition of size $n+k$. First we note that (up to rotation) there is only one way to arrange the inner bricks of size $\mu$ along a cycle of size $n+k$. Starting with this arrangement of inner bricks, choose any of the $k+1$ segments at which an inner brick starts. From this point, add outer bricks of sizes $\alpha_1, \alpha_2, \dots$  in a clockwise fashion. Now look at the word in the alphabet $\{B<L<N<U\}$ obtained by reading clockwise from the previously chosen starting point. It may or may not be a Lyndon word. But we claim there is always a unique way to chop it into Lyndon words so that the pieces are in reverse lexicographic order. Chop it like so and take the bi-brick permutation corresponding to this multiset of Lyndon words. For example, when $\mu=(5,1^3)$ and $\alpha=(2,4,2)$, we obtain the first four bi-brick permutations of Figure~\ref{fig:bigex}. Clearly each of these $k+1$ bi-brick permutations $\Pi$ have $\alpha(\Pi)=\alpha$ and $\mu(\Pi)=\mu$. Furthermore, each cycle's Lyndon word contains a $B$. This is because the original large cycle's (not necessarily Lyndon) word started with $B$, each cut was only made before another $B$.

We claim that these are the only bi-brick permutations $\Pi$ satisfying $\alpha(\Pi)=\alpha$ and $\mu(\Pi)=\mu$ in which each cycle's Lyndon word contains a $B$. We also claim that all the remaining bi-brick permutations $\Pi$ with $\alpha(\Pi)=\alpha$ and $\mu(\Pi)=\mu$ are obtained by rotating the inner bricks of a single bi-brick permutation. In particular, let $r$ and $s$ be as small as possible so that $\sum_{i=1}^r \alpha_i = n+s$. Then the desired bi-brick permutation's inner brick of size $n$ will lie on a cycle with outer bricks of size $\alpha_1, \dots, \alpha_r$, all of whose $s$ inner bricks of size $1$ will be underneath the outer brick corresponding to $\alpha_r$. Call this the main cycle. The remaining outer bricks of sizes $\alpha_{r+1},\dots,\alpha_{\ell(\alpha)}$ will each have it's own cycle with inner bricks all of size $1$. Then starting with this special bi-brick permutation, we will rotate the inner bricks of the main cycle so that all $s$ inner bricks of size $1$ still lie underneath the outer brick corresponding to $\alpha_r$. This gives the remaining terms of $\qpoly_{n,k}$. Continuing our previous example, we obtain the last two bi-brick permutations of Figure~\ref{fig:bigex}.

\begin{remark}
In the above algorithm, a possible value of $\stat$ is the number of segments covered by the inner brick of size $n$ before (clockwise) the outer brick corresponding to $\alpha_r$ starts. Various modifications of this statistic have been tried for non-hook shapes without success. For example, when $\mu$ is a two-column shape we would hope that $\stat$ is simply the number of times an outer brick starts at the interior point of an inner two-brick (i.e., the number of $L$'s in the Lyndon word), but this is not the case.
\end{remark}

\bibliography{nablam}
\bibliographystyle{acm}

\end{document}